\documentclass[12pt]{amsart}
\usepackage[margin=1in]{geometry}
\usepackage{amssymb,amsfonts,amsmath,mathtools}
\usepackage{enumerate}
\usepackage{mathrsfs}
\usepackage[unicode]{hyperref}
\usepackage[normalem]{ulem}
\usepackage{color}

\usepackage{constants}
\usepackage{bbm}

\newcommand{\Sym}{\mathrm{Sym}}

\newtheorem{theorem}{Theorem}[section]
\newtheorem{corollary}[theorem]{Corollary}
\newtheorem{proposition}[theorem]{Proposition}
\newtheorem{lemma}[theorem]{Lemma}
\theoremstyle{remark}
\newtheorem*{remark}{Remark}

\numberwithin{equation}{section}

\newconstantfamily{abcon}{symbol=c}

\title[Exceptional zeros and Sato--Tate distributions]{Exceptional zeros of Rankin--Selberg $L$-functions and joint Sato--Tate distributions}
\author{Jesse Thorner}
\address{Department of Mathematics, University of Illinois, Urbana, IL 61801, USA}
\email{\href{mailto:jesse.thorner@gmail.com}{jesse.thorner@gmail.com}}

\begin{document}

\begin{abstract}
Let $\chi$ be an idele class character over a number field $F$, and let $\pi,\pi'$ be non-dihedral twist-inequivalent cuspidal automorphic representations of $\mathrm{GL}_2(\mathbb{A}_F)$.  We prove that if $m,n\geq 0$ are integers, $m+n\geq 1$, $F$ is totally real, $\chi$ corresponds with a ray class character, and $\pi,\pi'$ correspond with primitive non-CM holomorphic Hilbert cusp forms, then the Rankin--Selberg $L$-function $L(s,\mathrm{Sym}^m(\pi)\times(\mathrm{Sym}^n(\pi')\otimes\chi))$ has a standard zero-free region with no exceptional Landau--Siegel zero.  This is new even for $F=\mathbb{Q}$.  As an application, we establish the strongest known unconditional effective rates of convergence in the Sato--Tate distribution for $\pi$ and the joint Sato--Tate distribution for $\pi$ and $\pi'$.
\end{abstract}

\maketitle

\section{Introduction and statement of the main results}
\label{sec:intro}

Let $F$ be a number field, $\mathbb{A}_F$ be the ring of adeles over $F$, and $\mathfrak{F}_{n}$ be the set of cuspidal automorphic representations $\pi$ of $\mathrm{GL}_{n}(\mathbb{A}_F)$ with unitary central character $\omega_{\pi}$.  Let $C(\pi)$ be the analytic conductor of $\pi$ (see \eqref{eqn:AC_def}), $L(s,\pi)$ be the standard $L$-function, and $\widetilde{\pi}\in\mathfrak{F}_{n}$ be the contragredient of $\pi$.  The generalized Riemann hypothesis (GRH) predicts that if $\pi\in\mathfrak{F}_{n}$ and $\mathrm{Re}(s)>\frac{1}{2}$, then $L(s,\pi)\neq 0$.  Jacquet and Shalika \cite{JS} proved that if $\mathrm{Re}(s)\geq 1$, then $L(s,\pi)\neq 0$, extending the work of Hadamard and de la Vall{\'e}e Poussin for the Riemann zeta function.  Let $|\cdot|$ denote the idelic norm, and let $t\in\mathbb{R}$.  By replacing $\pi$ with $\pi\otimes|\cdot|^{it}$ and varying $t$, we find that it is equivalent to prove that for all $\pi\in\mathfrak{F}_{n}$ and $\sigma\geq 1$, we have that $L(\sigma,\pi)\neq 0$.  By \cite[Theorem 1.1]{Wattanawanichkul}, there exists an absolute and effectively computable constant $\Cl[abcon]{ZFR1}=\Cr{ZFR1}>0$ such that if $\pi\neq\widetilde{\pi}$, then
\begin{equation}
\label{eqn:ZFR}
L(\sigma,\pi)\neq 0,\qquad \sigma\geq  1-\Cr{ZFR1}/(n\log C(\pi)).
\end{equation}
If $\pi=\widetilde{\pi}$, then $L(\sigma,\pi)$ has at most one zero $\beta_{\pi}$ (necessarily simple) in the interval \eqref{eqn:ZFR}.  If no such zero exists, then $L(s,\pi)$ has {\it no exceptional zero} or {\it no Landau--Siegel zero}.


Given $(\pi,\pi')\in\mathfrak{F}_{n}\times\mathfrak{F}_{n'}$, let $L(s,\pi\times\pi')$ be the Rankin--Selberg $L$-function, as defined by Jacquet, Piatetski-Shapiro, and Shalika \cite{JPSS}.  Shahidi \cite{Shahidi} proved that if $\mathrm{Re}(s)\geq 1$, then $L(s,\pi\times\pi')\neq 0$.  Replacing $\pi$ with $\pi\otimes|\cdot|^{it}$ and varying $t\in\mathbb{R}$, we find that it suffices to prove that if $(\pi,\pi')\in\mathfrak{F}_{n}\times\mathfrak{F}_{n'}$ and $\sigma\geq 1$, then $L(\sigma,\pi\times\pi')\neq 0$.  The only fully uniform improvement is Brumley's narrow zero-free region (\cite{Brumley} and \cite[Appendix]{Lapid})---there exist effectively computable constants $\Cl[abcon]{Brumley}=\Cr{Brumley}(n,n',F,\varepsilon)>0$ and $\Cl[abcon]{Brumley2}=\Cr{Brumley2}(n,n',F,\varepsilon)>0$ such that
\[
L(\sigma,\pi\times\pi')\neq 0,\qquad \sigma\geq 1-\Cr{Brumley}/(C(\pi)C(\pi'))^{\Cr{Brumley2}}.
\]
Harcos and the author \cite[Theorem 1.1]{HarcosThorner} proved that for all $\varepsilon>0$, there exists an ineffective constant $\Cl[abcon]{Siegel}=\Cr{Siegel}(\pi,\pi',\varepsilon)>0$ such that if $\chi\in\mathfrak{F}_{1}$, then
\[
L(\sigma,\pi\times(\pi'\otimes\chi))\neq 0,\qquad \sigma\geq 1-\Cr{Siegel}/C(\chi)^{\varepsilon}.
\]

Langlands conjectured that $L(s,\pi\times\pi')$ factors as a product of standard $L$-functions (i.e., $L(s,\pi\times\pi')$ is modular).  As of now, this is only known in special cases, most notably when $\pi\in\mathfrak{F}_{2}$ and $\pi'\in\mathfrak{F}_{2}\cup\mathfrak{F}_{3}$ \cite{KimShahidi,RamakrishnanWang}.  Hoffstein and Ramakrishnan \cite{HoffsteinRamakrishnan} proved that if all Rankin--Selberg $L$-functions are modular and $\pi \in \cup_{n=2}^{\infty}\mathfrak{F}_{n}$, then $L(s,\pi)$ has no zero in \eqref{eqn:ZFR}, eliminating exceptional zeros.  Modularity is known only in special cases, so unconditionally eliminating exceptional zeros remains a difficult and fruitful problem.  For a Rankin--Selberg $L$-function $L(s,\pi\times\pi')$, we say that $L(s,\pi\times\pi')$ has no exceptional zero if there exists an effectively computable constant $\Cl[abcon]{ZFR11}=\Cr{ZFR11}(n,n')>0$ such that
\[
L(\sigma,\pi\times\pi')\neq 0,\qquad \sigma \geq 1- \Cr{ZFR11}/\log(C(\pi)C(\pi')).
\]

Let $v$ be a place of $F$, and let $F_v$ be the completion of $F$ relative to $v$.  Given $\pi\in\mathfrak{F}_{2}$, we express $\pi$ as a restricted tensor product $\bigotimes_v\pi_v$ of smooth, admissible representations of $\mathrm{GL}_2(F_v)$.  If $m\geq 0$, then $\mathrm{Sym}^m\colon \mathrm{GL}_2(\mathbb{C})\to\mathrm{GL}_{m+1}(\mathbb{C})$ is the $(m+1)$-dimensional irreducible representation of $\mathrm{GL}_2(\mathbb{C})$ on symmetric tensors of rank $m$.  If $P(x,y)$ is a homogeneous degree $m$ polynomial in two variables and $g\in\mathrm{GL}_2(\mathbb{C})$, then $\mathrm{Sym}^m(g)\in\mathrm{GL}_{m+1}(\mathbb{C})$ is the matrix giving the change in coefficients of $P$ under the change of variables by $g$.  Let $\varphi_v$ be the two-dimensional representation of the Deligne--Weil group attached to $\pi_v$ and $\mathrm{Sym}^m(\pi_v)$ be the smooth admissible representation of $\mathrm{GL}_{m+1}(F_v)$ attached to the representation $\mathrm{Sym}^m\circ \varphi_v$.  By the local Langlands correspondence, $\mathrm{Sym}^m(\pi_v)$ is well-defined for every place $v$ of $F$.

Langlands conjectured that $\mathrm{Sym}^m(\pi) = \bigotimes_v \mathrm{Sym}^m(\pi_v)$ is an automorphic representation of $\mathrm{GL}_{m+1}(\mathbb{A}_F)$, possibly not cuspidal.  This is known for $m\leq 4$ \cite{GJ,Kim,KimShahidi}.  When $F$ is totally real, much more is known for $\pi$ lying in the set
\[
\mathrm{RA}_{F}=\{\pi\in\mathfrak{F}_{2}\colon \textup{$\pi$ non-dihedral and regular algebraic}\}
\]
(see \cite[Section 7]{BLGHT}).  Such $\pi$ correspond with primitive non-CM holomorphic Hilbert cusp forms whose weights are integers at least 2.  When $F=\mathbb{Q}$, these are the classical non-CM holomorphic newforms.  Newton and Thorne \cite{NewtonThorne,NewtonThorne2,NewtonThorne3} proved that if $m\geq 0$, $F$ is totally real, and $\pi\in\mathrm{RA}_{F}$, then $\mathrm{Sym}^m(\pi)\in\mathfrak{F}_{m+1}$.  At least as early as \cite{CogdellMichel}, it has been known that if $\mathrm{Sym}^j(\pi)\in\mathfrak{F}_{j+1}$ for all $j\leq n+2$, then $L(s,\Sym^n(\pi))$ has no exceptional zero; see \cite{Thorner_SatoTate3}.

Given $\pi,\pi'\in\mathfrak{F}_{n}$, we write $\pi\sim\pi'$ if there exists $\psi\in\mathfrak{F}_{1}$ such that $\pi'=\pi\otimes\psi$.  Otherwise, we write $\pi\not\sim\pi'$.  Let $\mathbbm{1}\in\mathfrak{F}_{1}$ be the trivial representation, whose $L$-function is the Dedekind zeta function $\zeta_F(s)$.  If $\pi,\pi'\in\mathrm{RA}_F$ and $\pi\not\sim\pi'$, then it is unclear whether the cuspidal automorphy of the symmetric powers of $\pi$ and $\pi'$ suffices to ensure that the Rankin--Selberg $L$-function $L(s,\Sym^m(\pi)\times\Sym^{n}(\pi'))$ has no exceptional zero.  In particular, the approach to establishing zero-free regions without exceptional zeros in \cite{Banks,HoffsteinLockhart,HoffsteinRamakrishnan,Luo,RamakrishnanWang} fails without knowing the modularity of certain Rankin--Selberg $L$-functions beyond the reach of current methods.

To state our main result, we define
\begin{equation}
\label{eqn:star_def}
\mathfrak{F}_{n}^*=\{\pi\in\mathfrak{F}_{n}\colon \omega_{\pi}\textup{ is trivial on the diagonally embedded positive reals}\}.	
\end{equation}
For all $\pi\in\mathfrak{F}_{n}$, there exists a unique pair $(\pi^*,t_{\pi})\in\mathfrak{F}_{n}^*\times\mathbb{R}$ such that $\pi=\pi^*\otimes|\cdot|^{it_{\pi}}$.  It follows that $L(s,\pi)=L(s+it_{\pi},\pi^*)$.

\begin{theorem}
\label{thm:main}
Let $m,n\geq0$ be integers, $M=\max\{m,n\}\geq 1$, and $F$ be a totally real number field.  Let $\pi,\pi'\in\mathrm{RA}_F$ and $\chi=\chi^*|\cdot|^{it_{\chi}}\in\mathfrak{F}_{1}$.  Define
\begin{equation}
\label{eqn:y_def}
y_{\pi}=\begin{cases}
	0&\parbox{.55\textwidth}{if $\pi$ has squarefree conductor or $\pi$ corresponds with a non-CM elliptic curve over $\mathbb{Q}$,}\vspace{1mm}\\
	2&\mbox{otherwise,}
\end{cases}
\end{equation}
and similarly for $\pi'$.  Let $\chi^*$ correspond with a ray class character.  Define
\[
\mathcal{L}(s) =\begin{cases}
\displaystyle\frac{L(s,\mathrm{Sym}^m(\pi)\times(\mathrm{Sym}^n(\pi')\otimes\chi))}{L(s,\chi\omega_{\pi}^{m}\psi^{m})}&\mbox{if $m=n$, $\psi\in\mathfrak{F}_{1}$, and $\pi'=\pi\otimes\psi$,}\vspace{1mm}\\
L(s,\mathrm{Sym}^m(\pi)\times(\mathrm{Sym}^n(\pi')\otimes\chi))&\mbox{if $m\neq n$ or $\pi\not\sim\pi'$.}
\end{cases}
\]
There exists an absolute and effectively computable constant $\Cl[abcon]{main}>0$ such that
\[
\mathcal{L}(\sigma)\neq 0,\qquad \sigma\geq 1-\frac{\Cr{main}}{M^2[M^{\max\{y_{\pi},y_{\pi'}\}}\log(M C(\pi)C(\pi'))+\log C(\chi)]}.
\]
\end{theorem}
\begin{remark}
Theorem \ref{thm:main} is new when $m+n\geq 4$. We restrict to totally real $F$ only to ensure the cuspidal automorphy of all symmetric powers of all $\pi\in\mathrm{RA}_F$ via \cite{NewtonThorne3}.  In light of \eqref{eqn:AC_def}, Theorem \ref{thm:main} is uniform in $F$.
\end{remark}
Our next result follows immediately from Theorem \ref{thm:main}, replacing $\chi$ with $\chi|\cdot|^{it}$.
\begin{corollary}
\label{cor:big_ZFR2}
Let $m,n\geq 0$ be integers, $M=\max\{m,n\}\geq 1$, $F$ be a totally real number field, and $\pi,\pi'\in\mathrm{RA}_F\cap\mathfrak{F}_{2}^*$.  Let $y_{\pi}$ and $y_{\pi'}$ be as in \eqref{eqn:y_def}.  Let $\chi\in\mathfrak{F}_{1}^*$ correspond with a ray class character.  If $\pi\not\sim\pi'$ and $t\in\mathbb{R}$, then $L(\sigma+it,\mathrm{Sym}^m(\pi)\times(\mathrm{Sym}^n(\pi')\otimes\chi))\neq 0$ when
\[
\sigma\geq 1-\frac{\Cr{main}}{M^2[M^{\max\{y_{\pi},y_{\pi'}\}}\log(M C(\pi)C(\pi'))+\log(C(\chi)(3+|t|)^{[F:\mathbb{Q}]})]}.
\]
\end{corollary}
\begin{remark}
Our proof shows that if $m=0$, then $C(\pi)$ can be replaced by $1$ and $y_{\pi}$ by $0$.  If $n=0$, then $C(\pi')$ can be replaced by $1$ and $y_{\pi'}$ by $0$.
\end{remark}

Theorem \ref{thm:main} relies on the careful construction of an certain auxiliary Dirichlet series with non-negative Dirichlet coefficients.  The usual method for constructing such Dirichlet series, by taking the Rankin--Selberg convolution of a carefully chosen isobaric automorphic representation with its dual, underlies all of the proofs in \cite{Banks,HoffsteinLockhart,HoffsteinRamakrishnan,Luo,RamakrishnanWang}.  This method is not available for the proof of Theorem \ref{thm:main} in the most important case, when $\pi\not\sim\pi'$, because of insufficient progress towards the modularity of Rankin--Selberg $L$-functions (see Section \ref{subsec:earlier_work}).  The novelty in our proofs, therefore, is the construction of the auxiliary Dirichlet series in \eqref{eqn:D_def} and the unconditional verification of its useful analytic properties.


\subsection*{Acknowledgements}

The author thanks Farrell Brumley, Gergely Harcos, and Jeffrey Hoffstein for helpful conversations and Laurent Clozel for providing the strategy for Lemma \ref{lem:Rajan_multiplicity_one}.  The author is partially funded by the Simons Foundation.

\section{Application to Sato--Tate and joint Sato--Tate distributions}
\label{sec:ST}

Let $F$ be a totally real number field.  Let $v\nmid\infty$ denote a non-archimedean place of $F$ and $v\mid\infty$ otherwise.  Let $\pi=\bigotimes_v \pi_v\in\mathrm{RA}_{F}\cap \mathfrak{F}_{2}^*$, and let $S_{\pi}=\{v\nmid\infty\colon\textup{$\pi_v$ ramified}\}$.  Since $\pi\in\mathrm{RA}_{F}$, the central character $\omega_{\pi}=\bigotimes_v \omega_{\pi,v}$ corresponds with a ray class character over $F$.  For $v\nmid\infty$, let $q_v$ be the cardinality of the residue field of the local ring of integers $\mathcal{O}_v\subseteq F_v$ and $\varpi_v$ be a uniformizer.  Fix a root of unity $\zeta$ such that $\zeta^2\in\mathrm{im}(\omega_{\pi})$.  Let $v\nmid\infty$ satisfy $v\notin S_{\pi}$ and $\omega_{\pi,v}(\varpi_v)=\zeta^2$.  Let $a_{\pi}(v)$ be the eigenvalue of the Hecke operator
\[
\Big[\mathrm{GL}_2(\mathcal{O}_{F_v})\Big(\begin{matrix}\varpi_v&0\\0&1\end{matrix}\Big)\mathrm{GL}_2(\mathcal{O}_{F_v})\Big]
\]
on $\pi_v^{\mathrm{GL}_2(\mathcal{O}_{F_v})}$.  Since $\mathrm{Sym}^m(\pi)\in\mathfrak{F}_{m+1}$ for all $m\geq 0$ \cite{NewtonThorne3}, we have that
\[
\frac{a_{\pi}(v)}{2\zeta}\in[-1,1].
\]

The Sato--Tate conjecture for $\pi$ states that if $\mathrm{d}\mu_{\mathrm{ST}}=(2/\pi)\sqrt{1-t^2}~\mathrm{d}t$ is the Sato--Tate measure, $I\subseteq[-1,1]$ is a fixed subinterval, and $\mathbf{1}_{I}(t)$ is the indicator function of $I$, then
\begin{equation}
\label{eqn:ST1}
\lim_{x\to\infty}\sum_{\substack{v\notin S_{\pi},~q_v\leq x \\ \omega_{\pi,v}(\varpi_v)=\zeta^2}}\mathbf{1}_I\Big(\frac{a_{\pi}(v)}{2\zeta}\Big)\Big\slash\sum_{\substack{v\notin S_{\pi},~q_v\leq x \\ \omega_{\pi,v}(\varpi_v)=\zeta^2}}1=\mu_{\mathrm{ST}}(I).
\end{equation}
This was proved when $F=\mathbb{Q}$ by Barnet-Lamb, Geraghty, Harris, and Taylor \cite{BLGHT} and for all totally real $F$ by Barnet-Lamb, Gee, and Geraghty \cite{BLGG}.


Let $\pi,\pi'\in\mathrm{RA}_{F}\cap\mathfrak{F}_{2}^*$.  Let $I,I'\subseteq[-1,1]$ be fixed subintervals, and let $\zeta$ (resp. $\zeta'$) be a fixed root of unity such that $\zeta^2\in\mathrm{im}(\omega_{\pi})$ (resp. ${\zeta'}^2\in\mathrm{im}(\omega_{\pi'})$).  Generalizing a question of Katz and Mazur originally posed for non-CM elliptic curves, one might ask whether
\begin{equation}
\label{eqn:ST2}
\lim_{x\to\infty}\sum_{\substack{v\notin S_{\pi}\cup S_{\pi'},~q_v\leq x \\ \omega_{\pi,v}(\varpi_v)=\zeta^2 \\ \omega_{\pi',v}(\varpi_v)=\zeta'^2 }}\mathbf{1}_I\Big(\frac{a_{\pi}(v)}{2\zeta}\Big)\mathbf{1}_{I'}\Big(\frac{a_{\pi'}(v)}{2\zeta'}\Big)\Big\slash\sum_{\substack{v\notin S_{\pi}\cup S_{\pi'},~q_v\leq x \\ \omega_{\pi,v}(\varpi_v)=\zeta^2 \\ \omega_{\pi',v}(\varpi_v)=\zeta'^2}}1 = \mu_{\mathrm{ST}}(I)\mu_{\mathrm{ST}}(I').
\end{equation}
Building on work of Harris \cite{Harris}, Wong \cite{Wong} proved \eqref{eqn:ST2} when $\omega_{\pi}=\omega_{\pi'}=\mathbbm{1}$ and $\pi\not\sim\pi'$.  Using Theorem \ref{thm:main}, we prove variants of \eqref{eqn:ST1} and \eqref{eqn:ST2} with effective rates of convergence.
\begin{theorem}
\label{thm:ST}
Let $F$ be a totally real number field and $\pi,\pi'\in\mathrm{RA}_{F}\cap\mathfrak{F}_{2}^*$.  Let $\zeta$ (resp. $\zeta'$) be a fixed root of unity such that $\zeta^2$ (resp. $\zeta'^2$) lies in the image of $\omega_{\pi}$ (resp. $\omega_{\pi'}$).  Let $S_{\pi}$ (resp. $S_{\pi'}$) be the set of places $v\nmid\infty$ of $F$ such that $\pi_v$ (resp. $\pi_v'$) is ramified.  Let $y_{\pi}$ and $y_{\pi'}$ be as in \eqref{eqn:y_def}.  There exists an absolute, effectively computable constant $\Cl[abcon]{ST1}>0$ such that the following results are true.
\begin{enumerate}
	\item If $I\subseteq[-1,1]$ is a closed subinterval and $x\geq 2$, then
\begin{align*}
\Big|\sum_{\substack{v\notin S_{\pi},~q_v\leq x \\ \omega_{\pi,v}(\varpi_v)=\zeta^2}}\mathbf{1}_I\Big(\frac{a_{\pi}(v)}{2\zeta}\Big)\Big\slash\sum_{\substack{v\notin S_{\pi},~q_v\leq x \\ \omega_{\pi,v}(\varpi_v)=\zeta^2 }}1-\mu_{\mathrm{ST}}(I)\Big|\leq \Cr{ST1} \sqrt{[F:\mathbb{Q}]}\Big(\frac{\log(C(\pi)\log x)}{\sqrt{\log x}}\Big)^{\frac{2}{2+y_{\pi}}}.
\end{align*}
\item Assume that $\pi\not\sim\pi'$.  Let $I,I'\subseteq[-1,1]$ be closed subintervals and $x\geq 2$.  If there are infinitely many $v\nmid\infty$ such that $\omega_{\pi,v}(\varpi_v)=\zeta^2$ and $\omega_{\pi',v}(\varpi_v)=\zeta'^2$, then
\begin{align*}
&\Big|\sum_{\substack{v\notin S_{\pi}\cup S_{\pi'},~q_v\leq x \\ \omega_{\pi,v}(\varpi_v)=\zeta^2 \\ \omega_{\pi',v}(\varpi_v)=\zeta'^2}}\mathbf{1}_I\Big(\frac{a_{\pi}(v)}{2\zeta}\Big)\mathbf{1}_{I'}\Big(\frac{a_{\pi'}(v)}{2\zeta'}\Big)\Big\slash\sum_{\substack{v\notin S_{\pi}\cup S_{\pi'},~q_v\leq x \\ \omega_{\pi,v}(\varpi_v)=\zeta^2 \\ \omega_{\pi',v}(\varpi_v)=\zeta'^2}}1-\mu_{\mathrm{ST}}(I)\mu_{\mathrm{ST}}(I')\Big|\\
&\leq \Cr{ST1} \sqrt{[F:\mathbb{Q}]}\Big(\frac{\log(C(\pi)C(\pi')\log x)}{\sqrt{\log x}}\Big)^{\frac{2}{2+\max\{y_{\pi},y_{\pi'}\}}}.
\end{align*}
\end{enumerate}
\end{theorem}
\begin{remark}
In \cite[Theorem 1.1]{Thorner_SatoTate3}, the author proved Theorem \ref{thm:ST}(1) when $F=\mathbb{Q}$ and $\omega_{\pi}=\mathbbm{1}$.  Theorem \ref{thm:ST}(1) is new in all other cases.  Theorem \ref{thm:ST}(2) is new in all cases; in \cite[Theorem 1.2]{Thorner_SatoTate3}, where $F=\mathbb{Q}$ and $\omega_{\pi}=\omega_{\pi'}=\mathbbm{1}$, the author proved a version of Theorem \ref{thm:ST}(2) with an error term that only decays as a power of $\log\log x$ because of the potential existence of exceptional zeros.  These are now precluded by Theorem \ref{thm:main}.  Throughout \cite{Thorner_SatoTate3}, it was assumed that $y_{\pi}=y_{\pi'}=0$.

As in Theorem \ref{thm:main}, we restrict to totally real $F$ only to ensure the cuspidal automorphy of all symmetric powers of all $\pi\in\mathrm{RA}_F$ via \cite{NewtonThorne3}.  By our definition for $C(\pi)$ and $C(\pi')$ in \eqref{eqn:AC_def}, Theorem \ref{thm:ST} is uniform in  $F$.
\end{remark}
We will only detail the proof of Theorem \ref{thm:ST}(2), since Theorem \ref{thm:ST}(1) is easier to prove.

\section{Analytic properties of $L$-functions}
\label{sec:Properties}

Let $F$ be a number field with absolute discriminant $D_F$.  For each place $v$ of $F$, we write $v\mid\infty$ (resp. $v\nmid \infty$) if $v$ is archimedean (resp. non-archimedean).  If $v\nmid\infty$, then $q_v$ is the cardinality of the residue field of the local ring of integers $\mathcal{O}_v\subseteq F_v$, and $\varpi_v$ is the uniformizer.  The properties of $L$-functions given here rely on \cite{GodementJacquet,JPSS,MoeglinWaldspurger}.  In our use of $f=O(g)$ or $f\ll g$, the implied constant is always absolute and effectively computable.

Let $\pi\in\mathfrak{F}_{n}$.  Recall $\mathfrak{F}_{n}^*$ in \eqref{eqn:star_def}.  There exists a unique pair $(\pi^*,t_{\pi})\in\mathfrak{F}_{n}^*\times\mathbb{R}$ such that $\pi=\pi^*\otimes|\cdot|^{it_{\pi}}$, in which case $L(s,\pi)=L(s+it_{\pi},\pi^*)$.  Therefore, in our discussion of standard $L$-functions and Rankin--Selberg $L$-functions, it suffices to restrict to $\mathfrak{F}_{n}^*$.

\subsection{Standard $L$-functions}
\label{subsec:standard}

Let $\pi\in\mathfrak{F}_{n}^*$, let $\widetilde{\pi}\in\mathfrak{F}_{n}^*$ be the contragredient, and let $\omega_{\pi}$ the central character.  We express $\pi$ as a restricted tensor product $\bigotimes_v \pi_v$ of smooth admissible representations of $\mathrm{GL}_n(F_v)$.  Let $\delta_{\pi}=1$ if $\pi=\mathbbm{1}$ and $\delta_{\pi}=0$ otherwise.  Let
\[
S_{\pi}=\{v\nmid\infty\colon \textup{$\pi_v$ ramified}\},\qquad S_{\pi}^{\infty}=S_{\pi}\cup\{v\mid\infty\}.
\]
Let $N_{\pi}$ be the norm of the conductor of $\pi$.  If $v\nmid\infty$, then there are $n$ Satake parameters $(\alpha_{j,\pi}(v))_{j=1}^n$ such that
	\[
	L(s,\pi)=\prod_{v\nmid\infty}L(s,\pi_{v}),\qquad L(s,\pi_{v}) = \prod_{j=1}^{n}\frac{1}{1-\alpha_{j,\pi}(v)q_v^{-s}}
	\]
	converges absolutely for $\mathrm{Re}(s)>1$.  If $v\in S_{\pi}$, then at least one of the $\alpha_{j,\pi}(p)$ equals zero.
	
If $v\mid\infty$, then $(\mu_{j,\pi}(v))_{j=1}^n$ are the Langlands parameters at $v$, from which we define
\[
L(s,\pi_{\infty}) = \prod_{v\mid\infty}L(s,\pi_v) = \prod_{v\mid\infty}\prod_{j=1}^n \Gamma_v(s+\mu_{j,\pi}(v)),\qquad \Gamma_v(s)=\begin{cases}
	\pi^{-s/2}\Gamma(s/2)&\mbox{if $F_v=\mathbb{R}$,}\\
	2(2\pi)^{-s}\Gamma(s)&\mbox{if $F_v=\mathbb{C}$.}
	\end{cases}
	\]
	The completed $L$-function $\Lambda(s,\pi)=(s(1-s))^{\delta_{\pi}}(D_F^n N_{\pi})^{\frac{s}{2}}L(s,\pi)L(s,\pi_{\infty})$ is entire of order $1$.  Since, $\{\alpha_{j,\widetilde{\pi}}(v)\}=\{\overline{\alpha_{j,\pi}(v)}\}$, $N_{\widetilde{\pi}}=N_{\pi}$, and $\{\mu_{j,\widetilde{\pi}}(v)\}=\{\overline{\mu_{j,\pi}(v)}\}$, we have that $L(s,\widetilde{\pi})=\overline{L(\overline{s},\pi)}$.  The analytic conductor is
\begin{equation}
\label{eqn:AC_def}
C(\pi)=D_F^n N_{\pi} \prod_{v\mid\infty}\prod_{j=1}^n (|\mu_{j,\pi}(v)|+3)^{[F_v:\mathbb{R}]}.
\end{equation}
By \cite{LRS,MullerSpeh} there exists $\theta_n\in[0,1/2-1/(n^2+1)]$ such that
\begin{equation}
\label{eqn:Ramanujan1}
|\alpha_{j,\pi}(v)|\leq q_v^{\theta_{n}},\qquad \mathrm{Re}(\mu_{j,\pi}(v))\geq -\theta_{n}.
\end{equation}
We define $a_{\pi}(v^{\ell})$ by the Dirichlet series identity
\[
-\frac{L'}{L}(s,\pi)=\sum_{v}\sum_{\ell=1}^{\infty}\frac{\sum_{j=1}^n \alpha_{j,\pi}(v)^{\ell}\log q_v}{q_v^{\ell s}}=\sum_{v\nmid\infty}\sum_{\ell=1}^{\infty}\frac{a_{\pi}(v^{\ell})\log q_v}{q_v^{\ell s}}.
\]

\subsection{Rankin--Selberg $L$-functions}
\label{subsec:RS}

Let $\pi\in\mathfrak{F}_{n}^*$ and $\pi'\in\mathfrak{F}_{n'}^*$.  Let $\delta_{\pi\times\pi'}=1$ if $\pi'=\widetilde{\pi}$ and $\delta_{\pi\times\pi'}=0$ otherwise.  For each $v\notin S_{\pi}^{\infty}\cup S_{\pi'}^{\infty}$, define
\[
L(s,\pi_{v}\times\pi_{v}')=\prod_{j=1}^n \prod_{j'=1}^{n'}\frac{1}{1-\alpha_{j,\pi}(v)\alpha_{j',\pi'}(v)q_v^{-s}}.
\]
Jaquet, Piatetski-Shapiro, and Shalika proved the following theorem.

\begin{theorem}\cite{JPSS}
\label{thm:JPSS}
If $(\pi,\pi')\in\mathfrak{F}_{n}^*\times\mathfrak{F}_{n'}^*$, then there exist
\begin{enumerate}
\item complex numbers $(\alpha_{j,j',\pi\times\pi'}(v))_{j=1}^n{}_{j'=1}^{n'}$ for each $v\in S_{\pi}\cup S_{\pi'}$, from which we define
	\begin{align*}
	L(s,\pi_v\times\pi_v') = \prod_{j=1}^n \prod_{j'=1}^{n'}\frac{1}{1-\alpha_{j,j',\pi\times\pi'}(v)q_v^{-s}},\quad L(s,\widetilde{\pi}_v\times\widetilde{\pi}_v') = \prod_{j=1}^n \prod_{j'=1}^{n'}\frac{1}{1-\overline{\alpha_{j,j',\pi\times\pi'}(v)}q_v^{-s}};
	\end{align*}
\item complex numbers $(\mu_{j,j',\pi\times\pi'}(v))_{j=1}^{n}{}_{j'=1}^{n'}$ for each $v\mid\infty$, from which we define
	\begin{align*}
	L(s,\pi_{v}\times\pi_{v}') = \prod_{j=1}^n \prod_{j'=1}^{n'}\Gamma_v(s+\mu_{j,j',\pi\times\pi'}(v)),\quad L(s,\widetilde{\pi}_{v}\times\widetilde{\pi}_{v}') = \prod_{j=1}^n \prod_{j'=1}^{n'}\Gamma_v(s+\overline{\mu_{j,j',\pi\times\pi'}(v)});
	\end{align*}
\item a conductor, an integral ideal whose norm is denoted $N_{\pi\times\pi'}=N_{\widetilde{\pi}\times\widetilde{\pi}'}$; and
\item a complex number $W(\pi\times\pi')$ of modulus $1$
\end{enumerate}
such that the Rankin--Selberg $L$-functions
\[
L(s,\pi\times\pi')=\prod_{v\nmid\infty}L(s,\pi_v\times\pi_v'),\qquad L(s,\widetilde{\pi}\times\widetilde{\pi}')=\prod_{v\nmid\infty}L(s,\widetilde{\pi}_v\times\widetilde{\pi}_v')
\]
converge absolutely for $\mathrm{Re}(s)>1$, the completed $L$-functions
\begin{align*}
\Lambda(s,\pi\times\pi') &= (s(1-s))^{\delta_{\pi\times\pi'}} (D_F^{n'n}N_{\pi\times\pi'})^{\frac{s}{2}} L(s,\pi\times\pi')\prod_{v\mid\infty}L(s,\pi_v\times\pi_{v}')\\
\Lambda(s,\widetilde{\pi}\times\widetilde{\pi}') &= (s(1-s))^{\delta_{\pi\times\pi'}} (D_F^{n'n}N_{\pi\times\pi'})^{\frac{s}{2}}L(s,\widetilde{\pi}\times\widetilde{\pi}')\prod_{v\mid\infty}L(s,\widetilde{\pi}_{v}\times\widetilde{\pi}_{v}')
\end{align*}
are entire of order $1$, and $\Lambda(s,\pi\times\pi')=W(\pi\times\pi')\Lambda(1-s,\widetilde{\pi}\times\widetilde{\pi}')$.
\end{theorem}

It follows from Theorem \ref{thm:JPSS} that
\begin{equation}
	\label{eqn:dual}
	L(s,\widetilde{\pi}\times\widetilde{\pi}')=\overline{L(\overline{s},\pi\times\pi)}.
	\end{equation}
We can explicitly determine the numbers $\mu_{j,j',\pi\times\pi'}(v)$ at all $v\mid\infty$ (resp. $\alpha_{j,j',\pi\times\pi'}(v)$ at all $v\in S_{\pi}\cup S_{\pi'}$) using the archimedean case of the local Langlands correspondence \cite[Section 3.1]{MullerSpeh} (resp. \cite[Appendix]{SoundararajanThorner}).  These descriptions and \eqref{eqn:Ramanujan1} yield the bounds
	\begin{equation}
		\label{eqn:Ramanujan2}
	|\alpha_{j,j',\pi\times\pi'}(v)|\leq  q_v^{\theta_n+\theta_{n'}},\qquad \mathrm{Re}(\mu_{j,j',\pi\times\pi'}(v))\geq -(\theta_n+\theta_{n'}).
	\end{equation}
	If $\ell\geq 1$ is an integer and $v\nmid\infty$, then we define
	\begin{equation}
	\label{eqn:a_def}
	a_{\pi\times\pi'}(v^{\ell})= \begin{cases}
 	a_{\pi}(v^{\ell})a_{\pi'}(v^{\ell})&\mbox{if $v\notin S_{\pi}\cup S_{\pi'}$,}\\
 	\sum_{j=1}^n \sum_{j'=1}^{n'}\alpha_{j,j',\pi\times\pi'}(v)^{\ell}&\mbox{if $v\in S_{\pi}\cup S_{\pi'}$,}
 \end{cases}\qquad a_{\widetilde{\pi}\times\widetilde{\pi}'}(v^{\ell})=\overline{a_{\pi\times\pi'}(v^{\ell})}.\hspace{-0.5mm}
	\end{equation}
	We have the Dirichlet series identity
	\begin{equation}
		\label{eqn:log_deriv}
	-\frac{L'}{L}(s,\pi\times\pi')=\sum_{v\nmid\infty}\sum_{\ell=1}^{\infty}\frac{a_{\pi\times\pi'}(v^{\ell})\log q_v}{q_v^{\ell s}},\qquad\mathrm{Re}(s)>1.
	\end{equation}
	The numbers $a_{\pi\times\widetilde{\pi}}(v^{\ell})$ and $a_{\pi'\times\widetilde{\pi}'}(v^{\ell})$ are non-negative (see Lemma \ref{lem:HR} below).  By \cite[Appendix]{SoundararajanThorner}, we have the bound
	\begin{equation}
	\label{eqn:Brumley}
	|a_{\pi\times\pi'}(v^{\ell})|^2 \leq a_{\pi\times\widetilde{\pi}}(v^{\ell})a_{\pi'\times\widetilde{\pi}'}(v^{\ell}).
	\end{equation}
	
	Let $\pi=\pi^*\otimes|\cdot|^{it_{\pi}}\in\mathfrak{F}_{n}$, $\pi'={\pi'}^*\otimes|\cdot|^{it_{\pi'}}\in\mathfrak{F}_{n'}$, and $\chi=\chi^*\otimes|\cdot|^{it_{\chi}}\in\mathfrak{F}_{1}$.  We define the analytic conductor
	\[
	C(\pi\times\pi')=D_F^{n'n}N_{\pi^*\times{\pi^*}'}\prod_{v\mid\infty}\prod_{j=1}^n \prod_{j=1}^{n'} (|\mu_{j,j',\pi^*\times{\pi^*}'}(v)+i(t_{\pi}+t_{\pi'})|+3)^{[F_v:\mathbb{Q}]}.
	\]
	Using the bound $N_{\pi^*\times{\pi'}^*}\mid N_{\pi^*}^{n'}N_{{\pi'}^*}^{n}$ \cite[Theorem 2]{BushnellHenniart} and the archimedean case of the local Langlands correspondence \cite[Section 3]{MullerSpeh} (see also \cite[Appendix]{Humphries}), we find that
	\begin{equation}
		\label{eqn:BH}
		\log C(\pi\times(\pi'\otimes\chi))\ll n'\log C(\pi)+n\log C(\pi')+n'n C(\chi).
	\end{equation}

\begin{lemma}
	\label{lem:GHL}
Let $J\geq 1$.  For $j\in\{1,\ldots,J\}$, let $(\pi_j,\pi_j',\chi_j)\in\mathfrak{F}_{n_j}\times\mathfrak{F}_{n_j'}\times\mathfrak{F}_{1}$.  Define
\[
\mathfrak{Q} = \prod_{j=1}^J C(\pi_j)^{n_j'}C(\pi_j')^{n_j}C(\chi)^{n_j'n_j},\quad \mathfrak{S}= \bigcup_{j=1}^J (S_{\pi_j}\cup S_{\pi_j'}\cup S_{\chi}),\quad D(s) =  \prod_{j=1}^J L(s,\pi_j\times(\pi'_j\otimes\chi_j)).
\]
Assume that $D(s)$ is holomorphic on $\mathbb{C}-\{1\}$ with a pole of order $r\geq 1$ at $s=1$.  Write
\begin{equation}
\label{eqn:aDdef}
a_D(v^{\ell})=\sum_{j=1}^J a_{\pi_j\times(\pi_j'\otimes\chi}(v^{\ell}),\qquad -\frac{D'}{D}(s) = \sum_{v}\sum_{\ell=1}^{\infty}\frac{a_D(v^{\ell})\log q_v}{q_v^{\ell s}}.
\end{equation}
Let $Q \geq \mathfrak{Q}$.  If $\mathrm{Re}(a_D(v^{\ell}))\geq 0$ for all $\ell\geq 1$ and $v\notin\mathfrak{S}$, then there exists an absolute and effectively computable constant $\Cl[abcon]{GHL}>0$ such that $D(\sigma)$ has at most $r$ zeros in the interval
\[
\Big[1-\frac{\Cr{GHL}}{r(\log Q+\sum_{v\in\mathfrak{S}}\frac{|a_D(v^{\ell})|\log q_v}{q_v})},1\Big)
\]
and no zeros in the interval $[1,\infty)$.
\end{lemma}
\begin{proof}
The proof is identical to \cite[Lemma 5.9]{IK} except that we do not trivially bound the contribution from the $a_D(v^{\ell})$ with $v\in \mathfrak{S}$.
\end{proof}

\subsection{Isobaric sums}
\label{subsec:isobaric}

Let $r\geq 1$ be an integer.  For $1\leq j\leq r$, let $\pi_{j}\in\mathfrak{F}_{d_j}$.  Langlands associated to $(\pi_1,\ldots,\pi_r)$ an automorphic representation of $\mathrm{GL}_{d_1+\cdots+d_r}(\mathbb{A}_F)$, the isobaric sum $\Pi=\pi_1\boxplus\cdots\boxplus\pi_r$.  Its $L$-function is $L(s,\Pi)=\prod_{j=1}^r L(s,\pi_j)$, and its contragredient is $\widetilde{\pi}_1\boxplus\cdots\boxplus\widetilde{\pi}_r$. Let $\mathfrak{A}_{n}$ be the set of isobaric automorphic representations of $\mathrm{GL}_n(\mathbb{A}_F)$.  If $\Pi=\pi_1\boxplus\cdots\boxplus\pi_r$ and $\Pi'=\pi_1'\boxplus\cdots\boxplus\pi_{r'}'$, then
\[
L(s,\Pi\times\Pi')=\prod_{j=1}^r \prod_{k=1}^{r'} L(s,\pi_j\times\pi_k').
\]
\begin{lemma}\cite{HoffsteinRamakrishnan}
\label{lem:HR}
If $\Pi\in\mathfrak{A}_{n}$, then $-\frac{L'}{L}(s,\Pi\times\widetilde{\Pi})$ has non-negative Dirichlet coefficients.
\end{lemma}

\subsection{Symmetric power lifts from $\mathrm{GL}_2$}

Let $\pi\in\mathfrak{F}_{2}$.  Define $\mathrm{Sym}^0(\pi)=\mathbbm{1}$ and $\mathrm{Sym}^1(\pi)=\pi$.  If is conjectured that if $m\geq 2$, then $\mathrm{Sym}^{m}(\pi)\in\mathfrak{A}_{m+1}$ with contragredient $\mathrm{Sym}^{m}(\widetilde{\pi})=\mathrm{Sym}^{m}(\pi)\otimes\overline{\omega}_{\pi}^{m}$.  This is known for $m\leq 4$ \cite{GJ,KimShahidi,Kim}.  If $\mathrm{Sym}^{m}(\pi)\in\mathfrak{A}_{m+1}$, then for each $v\in S_{\pi}$, there exist complex numbers $(\alpha_{j,\mathrm{Sym}^{m}(\pi)}(v))_{j=0}^{m}$ such that
\begin{equation}
\label{eqn:symm_power_def}
L(s,\mathrm{Sym}^m(\pi_v))^{-1}=\begin{cases}
\prod_{j=0}^m (1-\alpha_{1,\pi}(v)^{j}\alpha_{2,\pi}(v)^{m-j}q_v^{-s})&\mbox{if $v\nmid\infty$ and $v\notin S_{\pi}$,}\\
\prod_{j=0}^m(1-\alpha_{j,\mathrm{Sym}^{m}(\pi)}(v)q_v^{-s})&\mbox{if $v\nmid\infty$ and $v\in S_{\pi}$.}
\end{cases}	
\end{equation}
\begin{theorem}[\cite{NewtonThorne3}]
\label{thm:NT}
If $F$ is totally real, $j\geq 0$, and $\pi\in\mathrm{RA}_{F}$, then $\mathrm{Sym}^{j}(\pi)\in\mathfrak{F}_{j+1}$.
\end{theorem}

\begin{lemma}
\label{lem:CG}
Let $j,k,u\geq 0$ be integers, $\pi\in\mathfrak{F}_{2}$, and $\psi_1,\psi_2,\chi\in\mathfrak{F}_{1}$.  If $\mathrm{Sym}^u(\pi)\in\mathfrak{F}_{u+1}$ for $u\in\{j,k\}\cup\{j+k-2r\colon 0\leq r\leq \min\{j,k\}\}$ and
\[
\mathrm{Sym}^{j}(\pi\otimes\psi_1)\boxtimes\mathrm{Sym}^k(\pi\otimes\psi_2)\otimes\chi=\mathop{\boxplus}_{r=0}^{\min\{j,k\}}\mathrm{Sym}^{j+k-2r}(\pi)\otimes\chi\psi_1^j \psi_2^k\omega_{\pi}^{r}\in\mathfrak{A}_{(j+1)(k+1)},
\]
then
\[
L(s,\mathrm{Sym}^{j}(\pi\otimes\psi_1)\boxtimes(\mathrm{Sym}^k(\pi\otimes\psi_2)\otimes\chi))=L(s,\mathrm{Sym}^{j}(\pi\otimes\psi_2)\times(\mathrm{Sym}^k(\pi\otimes\psi_2)\otimes\chi)).
\]
In particular, if $v\nmid\infty$, then the following identities hold:
\begin{align}
a_{\mathrm{Sym}^{j}(\pi)\times(\mathrm{Sym}^k(\pi)\otimes\chi)}(v^{\ell}) &= \sum_{r=0}^{\min\{j,k\}}a_{\mathrm{Sym}^{j+k-2r}(\pi)\otimes\chi \omega_{\pi}^{r}}(v^{\ell}),\label{eqn:CG1}\\
a_{\mathrm{Sym}^{j}(\pi)\times\mathrm{Sym}^j(\widetilde{\pi})}(v^{\ell})&=1+\sum_{r=1}^j a_{\mathrm{Sym}^{2r}(\pi)\otimes\omega_{\pi}^{-r}}(v^{\ell}).\label{eqn:CG2}
\end{align}
\end{lemma}
\begin{proof}
	These follow from the Clebsch--Gordan identities.
\end{proof}

\begin{lemma}
	\label{lem:AC}
Let $F$ be totally real, let $\pi\in\mathrm{RA}_{F}$, and let $y_{\pi}$ be as in \eqref{eqn:y_def}. If $m\geq 1$, then $S_{\mathrm{Sym}^m(\pi)}\subseteq S_{\pi}$ and $\log C(\mathrm{Sym}^m(\pi))\ll m\log(m C(\pi)) + m^{1+y_{\pi}}\log N_{\pi}$.
\end{lemma}
\begin{proof}
Write the analytic conductor $C(\mathrm{Sym}^m(\pi))$ as $D_F^{m+1}N_{\mathrm{Sym}^m(\pi)}K_{\mathrm{Sym}^m(\pi)}$.  The bound $\log K_{\mathrm{Sym}^m(\pi)}\ll m\log(m K_{\pi})$ follows from the explicit description of $\pi_{v}$ at $v\mid\infty$ given by Moreno and Shahidi \cite{MorenoShahidi}.  If $y_{\pi}=0$, then the bound $\log N_{\mathrm{Sym}^m(\pi)}\ll m\log N_{\pi}$ follows from the local computations in \cite[Section 3]{CogdellMichel} and \cite[Appendix]{David}.  Otherwise, for $p$ a rational prime, it suffices to bound $\mathrm{ord}_p(N_{\mathrm{Sym}^m(\pi)})$.  By Lemma \ref{lem:CG} and the identity $\widetilde{\pi}=\pi\otimes\overline{\omega}_{\pi}$, the relation $L(s,\mathrm{Sym}^{m+1}(\pi)\times\widetilde{\pi})=L(s,\mathrm{Sym}^{m+2}(\pi)\otimes\overline{\omega}_{\pi})L(s,\mathrm{Sym}^m(\pi))$ holds.  Also, we have that $\mathrm{Sym}^{m+2}(\pi)=(\mathrm{Sym}^{m+2}(\pi)\otimes\overline{\omega}_{\pi})\otimes \omega_{\pi}$.  Bushnell and Henniart \cite{BushnellHenniart} proved that if $\pi\in\mathfrak{F}_{n}$ and $\pi'\in\mathfrak{F}_{n'}$, then $N_{\pi\times\pi'}\mid N_{\pi}^{n'}N_{\pi'}^n$.  
It follows that
\[
N_{\mathrm{Sym}^{m+2}(\pi)}\mid N_{\omega_{\pi}}^{m+3} N_{\mathrm{Sym}^{m+2}(\pi)\otimes\overline{\omega}_{\pi}}\qquad\textup{and}\qquad N_{\mathrm{Sym}^{m+2}(\pi)\otimes\overline{\omega}_{\pi}}N_{\mathrm{Sym}^m(\pi)}\mid N_{\mathrm{Sym}^{m+1}(\pi)}^2 N_{\pi}^{m+2}.
\]
Since $N_{\omega_{\pi}}\mid N_{\pi}$ \cite[Theorem A]{RamakrishnanYang}, the bound
\[
\mathrm{ord}_p(N_{\mathrm{Sym}^{m+2}(\pi)})\leq 2\mathrm{ord}_p(N_{\mathrm{Sym}^{m+1}(\pi)})-\mathrm{ord}_{p}(N_{\mathrm{Sym}^{m}(\pi)})+(2m+5)\mathrm{ord}_p(N_{\pi})
\]
holds.  By induction on $m$, we conclude that $\mathrm{ord}_p(N_{\mathrm{Sym}^m(\pi)})\leq m^3\mathrm{ord}_p(N_{\pi})$, as desired.
\end{proof}

\begin{lemma}
\label{lem:GRC}
Let $F$ be totally real, and let $\ell\geq 1$ and $m,n\geq 0$ be integers.  Let $\pi,\pi'\in\mathrm{RA}_F$.  If $v\nmid\infty$, then $|a_{\mathrm{Sym}^m(\pi)\times(\mathrm{Sym}^n(\pi')\otimes\chi)}(v^{\ell})|\leq (m+1)(n+1)$.
\end{lemma}
\begin{proof}
Using Theorem \ref{thm:NT}, we proceed as in \cite[Corollary 7.1.15]{potential} and conclude that at every place $v\nmid\infty$ (not just the unramified places) and every $\ell\geq 1$, we have that if $j,k\geq 0$, then
\begin{equation}
\label{eqn:GRC1}
|a_{\mathrm{Sym}^j(\pi)}(v^{\ell})|\leq j+1,\qquad |a_{\mathrm{Sym}^k(\pi')}(v^{\ell})|\leq k+1.
\end{equation}
Note that $a_{\mathrm{Sym}^m(\pi)\times\mathrm{Sym}^m(\widetilde{\pi})}(v^{\ell})\geq 0$ and $a_{\mathrm{Sym}^n(\pi')\times\mathrm{Sym}^n(\widetilde{\pi}')}(v^{\ell})\geq 0$ by Lemma \ref{lem:HR}, and
\[
|a_{\mathrm{Sym}^m(\pi)\times(\mathrm{Sym}^n(\pi)\otimes\chi)}(v^{\ell})|^2\leq a_{\mathrm{Sym}^m(\pi)\times\mathrm{Sym}^m(\widetilde{\pi})}(v^{\ell})a_{\mathrm{Sym}^n(\pi')\times\mathrm{Sym}^n(\widetilde{\pi}')}(v^{\ell})
\]
by \eqref{eqn:Brumley}.  Combining Lemma \ref{lem:CG} and \eqref{eqn:GRC1}, we find that
\[
a_{\mathrm{Sym}^m(\pi)\times\mathrm{Sym}^m(\widetilde{\pi})}(v^{\ell})\leq (m+1)^2,\qquad a_{\mathrm{Sym}^n(\pi')\times\mathrm{Sym}^n(\widetilde{\pi}')}(v^{\ell})\leq (n+1)^2.\qedhere
\]
\end{proof}

%
%

\section{Preliminaries for the proof of Theorem \ref{thm:main}}
\label{sec:strategy}

Let $(\pi_1,\pi_2)\in\mathfrak{F}_{n_1}\times\mathfrak{F}_{n_2}$.  Assume that $L(s,\pi_1\times\pi_2)$ is an irreducible $L$-function that is not the $L$-function of a real-valued ray class character.  The only known strategy to prove that $L(\sigma,\pi_1\times\pi_2)$ has no exceptional zero is to construct a Dirichlet series $D(s)$, integers $\ell_1,\ell_2\geq 0$ and $k\geq 1$ satisfying $\ell_1+\ell_2>k$, and a fixed $t\in(0,1)$ such that
\begin{enumerate}
	\item at each unramified $v\nmid\infty$ and $\ell\geq 1$, we have $\mathrm{Re}(a_D(v^{\ell}))\geq 0$ (see \eqref{eqn:aDdef}),
	\item $D(s)$ is holomorphic everywhere except for a pole of order $k$ at $s=1$, and
	\item $D(\sigma)(\sigma-1)^{k}L(\sigma,\pi_1\times\pi_2)^{-\ell_1}L(\sigma,\widetilde{\pi}_1\times\widetilde{\pi}_2)^{-\ell_2}$ is holomorphic at each $t<\sigma\leq 1$.
\end{enumerate}
A real zero of $L(s,\pi_1\times\pi_2)$ is a zero of $D(s)$ with multiplicity at least $\ell_1+\ell_2>k$, so the existence of a exceptional zero of $L(s,\pi_1\times\pi_2)$ contradicts Lemma \ref{lem:GHL} applied to $D(s)$.

\subsection{Earlier work}
\label{subsec:earlier_work}

First, let $\pi_1\in\cup_{n=2}^{\infty}\mathfrak{F}_{n}$ and $\pi_2=\mathbbm{1}$.  Hoffstein and Ramakrishnan \cite[Proof of Theorem B]{HoffsteinRamakrishnan} showed that sufficient progress towards the modularity of Rankin--Selberg $L$-functions suffices to prove that $L(s,\pi_1)$ has no exceptional zero.  Assume that $L(s,\pi_1\times\widetilde{\pi}_1)$ is modular so that there exists an isobaric automorphic representation $\pi_1\boxtimes\widetilde{\pi}_1$ such that $L(s,\pi_1\times\widetilde{\pi}_1)=L(s,\pi_1\boxtimes\widetilde{\pi}_1)$.  By \cite[Lemma 4.4]{HoffsteinRamakrishnan}, there exists a cuspidal constituent $\tau\notin\{\mathbbm{1},\pi_1,\widetilde{\pi}_1\}$ of $\pi_1\boxtimes\widetilde{\pi}_1$.  If $L(s,\pi_1\times\tau)$ is modular, then by \cite[Proofs of Lemma 4.4 and Claim 4.5]{HoffsteinLockhart}, $L(s,\pi_1\times\tau)/L(s,\pi_1)$ is entire.  Therefore, subject to the modularity of $L(s,\pi_1\times\widetilde{\pi}_1)$ and $L(s,\pi_1\times\tau)$, Hoffstein and Ramakrishnan prove that if $\Pi=\mathbbm{1}\boxplus\widetilde{\tau}\boxplus\pi_1$, then $D(s)=L(s,\Pi\times\widetilde{\Pi})$ satisfies the above criteria (1)--(3).  The results in \cite{Banks,HoffsteinLockhart,HoffsteinRamakrishnan,Luo,RamakrishnanWang} follow by proving the existence of $\pi_1\boxtimes\widetilde{\pi}_1$ and the existence of a cuspidal constituent $\tau$ of $\pi_1\boxtimes\widetilde{\pi}_1$ such that there exists an effectively computable constant $\Cl[abcon]{HR}=\Cr{HR}(n)\in(0,1)$ such that $L(\sigma,\pi_1\times\tau)/L(\sigma,\pi_1)$ is holomorphic at each $\sigma\in(\Cr{HR},1)$.

Let $F$ be totally real, $\pi,\pi'\in\mathrm{RA}_F$, $\chi\in\mathfrak{F}_{1}$, and $\pi\not\sim\pi'$.  Let $\pi_1=\mathrm{Sym}^m(\pi)$ and $\pi_2=\mathrm{Sym}^n(\pi')\otimes\chi$.  By Theorem \ref{thm:NT} and Lemma \ref{lem:CG}, $L(s,\pi_1\times\widetilde{\pi}_1)$ and $L(s,\pi_2\times\widetilde{\pi}_2)$ are modular.  By Lemma \ref{lem:CG} again, if $\tau=\mathrm{Ad}(\pi)$, then $L(s,\pi_1\times\tau)/L(s,\pi_1)$ is entire.  If there exists $\pi_0\in\mathfrak{F}_{(m+1)(n+1)}$ such that such that $L(s,\pi_1\times\pi_2)=L(s,\pi_0)$, then we can apply the Hoffstein--Ramakrishnan strategy to $\Pi = \mathbbm{1}\boxplus \mathrm{Ad}(\pi)\boxplus \pi_0$ and $D(s)=L(s,\Pi\times\widetilde{\Pi})$.  No such $\pi_0$ is known to exist yet for $m,n\geq 2$.  {\it Even if such a $\pi_0$ exists}, the zero-free region that follows from Lemma \ref{lem:GHL} applied to $D(s)$ is insufficient for Theorem \ref{thm:ST}(2).

\subsection{Ideas for our proofs}

Let $F$ be totally real and $m,n\geq 0$ satisfy $\max\{m,n\}\geq 1$.  Let $\pi,\pi'\in\mathrm{RA}_F$ and $\chi=\chi^*|\cdot|^{it_{\chi}}\in\mathfrak{F}_{1}$.  Suppose that $\pi\not\sim\pi'$.  If $j<0$ or $k<0$, and $\xi\in\mathfrak{F}_{1}$, then $L(s,\mathrm{Sym}^j(\pi)\times(\mathrm{Sym}^{k}(\pi')\otimes\xi))$ is the constant function $1$.  We introduce
\begin{equation}
\label{eqn:D_def}
\begin{aligned}
\mathcal{D}(s)&=L(s,\mathrm{Sym}^{m}(\pi)\times\mathrm{Sym}^m(\widetilde{\pi})) L(s,\mathrm{Sym}^{n}(\pi')\times\mathrm{Sym}^n(\widetilde{\pi}'))^2 L(s,\mathrm{Sym}^4(\pi)\otimes\overline{\omega}_{\pi}^{2})\\
&\cdot L(s,\mathrm{Sym}^{m}(\pi)\times(\mathrm{Sym}^{n}(\pi')\otimes\chi))^2 L(s,\mathrm{Sym}^m(\widetilde{\pi})\times(\mathrm{Sym}^n(\widetilde{\pi}')\otimes\overline{\chi}))^2 L(s,\mathrm{Ad}(\pi))^3\\
&\cdot L(s,\mathrm{Sym}^{m-2}(\pi)\times(\mathrm{Sym}^{n}(\pi')\otimes\chi\omega_{\pi})) L(s,\mathrm{Sym}^{m-2}(\widetilde{\pi})\times(\mathrm{Sym}^{n}(\widetilde{\pi}')\otimes\overline{\chi}\,\overline{\omega}_{\pi})) \\
&\cdot L(s,\mathrm{Sym}^{m+2}(\pi)\times(\mathrm{Sym}^{n}(\pi')\otimes\chi\overline{\omega}_{\pi})) L(s,\mathrm{Sym}^{m+2}(\widetilde{\pi})\times(\mathrm{Sym}^n(\widetilde{\pi}')\otimes\overline{\chi}\omega_{\pi}))\\
&\cdot \prod_{k=1}^n [L(s,\mathrm{Ad}(\pi)\times(\mathrm{Sym}^{2k}(\pi')\otimes\overline{\omega}_{\pi'}^{k}))^3L(s,\mathrm{Sym}^4(\pi)\times(\mathrm{Sym}^{2k}(\pi')\otimes \overline{\omega}_{\pi}^{2}\overline{\omega}_{\pi'}^{k}))].
\end{aligned}
\hspace{-3.03mm}
\end{equation}

\subsubsection{Non-negativity of Dirichlet coefficients}

Unlike in \cite{Banks,HoffsteinLockhart,HoffsteinRamakrishnan,Luo,RamakrishnanWang}, when $m\geq 2$ and $n\geq 2$, it is not yet known whether there exists an isobaric automorphic representation $\Pi$ such that $\mathcal{D}(s)=L(s,\Pi\times\widetilde{\Pi})$.  Therefore, we cannot use Lemma \ref{lem:HR} to establish the non-negativity of the Dirichlet coefficients $a_{\mathcal{D}}(v^{\ell})\log q_v$ of $-(\mathcal{D}'/D)(s)$ (recall \eqref{eqn:aDdef}).
\begin{lemma}
	\label{lem:coeff}
Let $\pi,\pi'\in\mathrm{RA}_{F}$ and $\chi\in\mathfrak{F}_{1}$.  If $v\notin S_{\pi}\cup S_{\pi'}\cup S_{\chi}$, then $a_{\mathcal{D}}(v^{\ell})\geq 0$.
\end{lemma}
\begin{proof}
Fix $\ell\geq 1$ and a place $v\nmid\infty$ of $F$ such that $v\notin S_{\pi}\cup S_{\pi'}\cup S_{\chi}$ We use \eqref{eqn:a_def} throughout the proof without further mention.  Since $v$ and $\ell$ are fixed, we suppress $(v^{\ell})$ throughout (e.g., $a_{\mathcal{D}}=a_{\mathcal{D}}(v^{\ell})$ and $a_{\mathrm{Ad}(\pi)}=a_{\mathrm{Ad}(\pi)}(v^{\ell})$). 
A direct calculation shows that
\begin{equation}
\label{eqn:aD0_formula}
\begin{aligned}
a_{\mathcal{D}}&=a_{\mathrm{Sym}^m(\pi)\times\mathrm{Sym}^m(\widetilde{\pi})}+2a_{\mathrm{Sym}^n(\pi')\times\mathrm{Sym}^n(\widetilde{\pi}')}+2a_{\mathrm{Sym}^m(\pi)\times(\mathrm{Sym}^n(\pi')\otimes\chi)}\\
&\qquad+2a_{\mathrm{Sym}^m(\widetilde{\pi})\times(\mathrm{Sym}^n(\widetilde{\pi}')\otimes\overline{\chi})}+a_{\mathrm{Sym}^{m-2}(\pi)\times(\mathrm{Sym}^{n}(\pi')\otimes\chi\omega_{\pi})}\\
&\qquad+a_{\mathrm{Sym}^{m-2}(\widetilde{\pi})\times(\mathrm{Sym}^{n}(\widetilde{\pi}')\otimes\overline{\chi}\,\overline{\omega}_{\pi})}+a_{\mathrm{Sym}^{m+2}(\pi)\times(\mathrm{Sym}^{n}(\pi')\otimes\chi\overline{\omega}_{\pi})}\\
&\qquad+a_{\mathrm{Sym}^{m+2}(\widetilde{\pi})\times(\mathrm{Sym}^n(\widetilde{\pi}')\otimes\overline{\chi}\omega_{\pi})}+a_{\mathrm{Sym}^4(\pi)\otimes\overline{\omega}_{\pi}^{2}}+3a_{\mathrm{Ad}(\pi)}\\
&\qquad+\sum_{k=1}^n [3a_{\mathrm{Ad}(\pi)\times(\mathrm{Sym}^{2k}(\pi')\otimes\overline{\omega}_{\pi'}^{k})}+ a_{\mathrm{Sym}^4(\pi)\times(\mathrm{Sym}^{2k}(\pi')\otimes \overline{\omega}_{\pi}^{2}\overline{\omega}_{\pi'}^{k})}].
\end{aligned}
\end{equation}
By Lemma \ref{lem:CG}, we find that $(a_{\mathrm{Ad}(\pi)}+1)^2 = 2+3a_{\mathrm{Ad}(\pi)}+a_{\mathrm{Sym}^4(\pi)\otimes\overline{\omega}_{\pi}^2}$ and
\begin{equation}
\label{eqn:CG_non-neg_2}
\begin{aligned}
&2a_{\mathrm{Sym}^n(\pi')\times\mathrm{Sym}^n(\widetilde{\pi}')}+3a_{\mathrm{Ad}(\pi)}+a_{\mathrm{Sym}^4(\pi)\otimes\overline{\omega}_{\pi}^{2}}\\
&\qquad+\sum_{k=1}^n [3a_{\mathrm{Ad}(\pi)\times (\mathrm{Sym}^{2k}(\pi')\otimes\overline{\omega}_{\pi'}^{k})}+a_{\mathrm{Sym}^4(\pi)\times (\mathrm{Sym}^{2k}(\pi')\otimes\overline{\omega}_{\pi}^{2} \overline{\omega}_{\pi'}^{k})}]\\
&=2a_{\mathrm{Sym}^n(\pi')\times\mathrm{Sym}^n(\widetilde{\pi}')}+(3a_{\mathrm{Ad}(\pi)}+a_{\mathrm{Sym}^4(\pi)\otimes\overline{\omega}_{\pi}^{2}})\sum_{k=0}^n a_{\mathrm{Sym}^{2k}(\pi')\otimes\overline{\omega}_{\pi'}^{k}}\\
&=2a_{\mathrm{Sym}^n(\pi')\times\mathrm{Sym}^n(\widetilde{\pi}')}+(3a_{\mathrm{Ad}(\pi)}+a_{\mathrm{Sym}^4(\pi)\otimes\overline{\omega}_{\pi}^{2}})a_{\mathrm{Sym}^n(\pi')\times\mathrm{Sym}^n(\widetilde{\pi}')}\\
&=(a_{\mathrm{Ad}(\pi)}+1)^2 |a_{\mathrm{Sym}^n(\pi')}|^2.
\end{aligned}
\end{equation}

Lemma \ref{lem:CG} yields $a_{\mathrm{Ad}(\pi)\times\mathrm{Sym}^m(\pi)} = a_{\mathrm{Sym}^{m-2}(\pi)\otimes\omega_{\pi}}+a_{\mathrm{Sym}^{m}(\pi)}+a_{\mathrm{Sym}^{m+2}(\pi)\otimes\overline{\omega}_{\pi}}$, hence
\begin{align*}
&a_{\mathrm{Sym}^{m-2}(\pi)\times(\mathrm{Sym}^n(\pi')\otimes\omega_{\pi}\chi)}+a_{\mathrm{Sym}^{m+2}(\pi)\times(\mathrm{Sym}^n(\pi')\otimes\overline{\omega}_{\pi}\chi)}+2a_{\mathrm{Sym}^m(\pi)\times(\mathrm{Sym}^n(\pi')\otimes\chi)}\\
&=(a_{\mathrm{Sym}^{m-2}(\pi)\otimes\omega_{\pi}}+a_{\mathrm{Sym}^{m}(\pi)}+a_{\mathrm{Sym}^{m+2}(\pi)\otimes\overline{\omega}_{\pi}})a_{\mathrm{Sym}^n(\pi')\otimes\chi}+a_{\mathrm{Sym}^m(\pi)\times(\mathrm{Sym}^n(\pi')\otimes\chi)}\\
&=a_{\mathrm{Ad}(\pi)\times\mathrm{Sym}^m(\pi)\otimes\chi}a_{\mathrm{Sym}^n(\pi')}+a_{\mathrm{Sym}^m(\pi)\times(\mathrm{Sym}^n(\pi')\otimes\chi)}\\
&=a_{\mathrm{Ad}(\pi)}a_{\mathrm{Sym}^m(\pi)}a_{\mathrm{Sym}^n(\pi')\otimes\chi}+a_{\mathrm{Sym}^m(\pi)}a_{\mathrm{Sym}^n(\pi')\otimes\chi}\\
&=(a_{\mathrm{Ad}(\pi)}+1)a_{\mathrm{Sym}^m(\pi)}a_{\mathrm{Sym}^n(\pi')\otimes\chi}.
\end{align*}
A dual calculation shows that
\begin{align*}
&a_{\mathrm{Sym}^{m-2}(\widetilde{\pi})\times(\mathrm{Sym}^n(\widetilde{\pi}')\otimes\overline{\omega}_{\pi}\overline{\chi})}+a_{\mathrm{Sym}^{m+2}(\widetilde{\pi})\times(\mathrm{Sym}^n(\widetilde{\pi}')\otimes\omega_{\pi}\overline{\chi})}+2a_{\mathrm{Sym}^m(\widetilde{\pi})\times(\mathrm{Sym}^n(\widetilde{\pi}')\otimes\overline{\chi})}\\
&=(a_{\mathrm{Ad}(\widetilde{\pi})}+1)\overline{a_{\mathrm{Sym}^m(\pi)}a_{\mathrm{Sym}^n(\pi')\otimes\chi}}.
\end{align*}
Since $\mathrm{Ad}(\pi)$ is self-dual, we have that $\mathrm{Ad}(\widetilde{\pi})=\mathrm{Ad}(\pi)$, $a_{\mathrm{Ad}(\pi)}\in\mathbb{R}$, and
\begin{equation}
\label{eqn:CG_non-neg_3}
\begin{aligned}
&a_{\mathrm{Sym}^{m-2}(\pi)\times(\mathrm{Sym}^n(\pi')\otimes\omega_{\pi}\chi)}+a_{\mathrm{Sym}^{m-2}(\widetilde{\pi})\times(\mathrm{Sym}^n(\widetilde{\pi}')\otimes\overline{\omega}_{\pi}\overline{\chi})}\\
&\qquad+2a_{\mathrm{Sym}^m(\pi)\times(\mathrm{Sym}^n(\pi')\otimes\chi)}+2a_{\mathrm{Sym}^m(\widetilde{\pi})\times(\mathrm{Sym}^n(\widetilde{\pi}')\otimes\overline{\chi})}\\
&\qquad+a_{\mathrm{Sym}^{m+2}(\widetilde{\pi})\times(\mathrm{Sym}^n(\widetilde{\pi}')\otimes\omega_{\pi}\overline{\chi})}+a_{\mathrm{Sym}^{m+2}(\pi)\times(\mathrm{Sym}^n(\pi')\otimes\overline{\omega}_{\pi}\chi)}\\
&=2\mathrm{Re}((a_{\mathrm{Ad}(\pi)}+1)a_{\mathrm{Sym}^m(\pi)}a_{\mathrm{Sym}^n(\pi')\otimes\chi}).
\end{aligned}
\end{equation}
Inserting \eqref{eqn:CG_non-neg_2} and \eqref{eqn:CG_non-neg_3} into \eqref{eqn:aD0_formula}, we find that $a_{\mathcal{D}}$ equals
\begin{align*}
&=a_{\mathrm{Sym}^m(\pi)\times\mathrm{Sym}^m(\widetilde{\pi})}+(a_{\mathrm{Ad}(\pi)}+1)^2 |a_{\mathrm{Sym}^{n}(\pi')}|^2+2\mathrm{Re}((a_{\mathrm{Ad}(\pi)}+1)a_{\mathrm{Sym}^m(\pi)}a_{\mathrm{Sym}^n(\pi')\otimes\chi})\\
&= |a_{\mathrm{Sym}^m(\pi)}|^2+(a_{\mathrm{Ad}(\pi)}+1)^2 |a_{\mathrm{Sym}^{n}(\pi')\otimes\chi}|^2+2\mathrm{Re}((a_{\mathrm{Ad}(\pi)}+1)a_{\mathrm{Sym}^m(\pi)}a_{\mathrm{Sym}^n(\pi')\otimes\chi}).
\end{align*}
This factors as $|\overline{a_{\mathrm{Sym}^m(\pi)}}+a_{\mathrm{Ad}(\pi)}a_{\mathrm{Sym}^n(\pi')\otimes\chi}+a_{\mathrm{Sym}^n(\pi')\otimes\chi}|^2\geq 0$.
\end{proof}

\subsubsection{Order the of pole at $s=1$}

Unlike in \cite{Banks,HoffsteinLockhart,HoffsteinRamakrishnan,Luo,RamakrishnanWang}, our $\mathcal{D}(s)$ in \eqref{eqn:D_def} has several factors for which the basic theory of Rankin--Selberg $L$-functions cannot determine whether a pole exists.  For certain choices of $m$ and $n$, there might be factors in $\mathcal{D}(s)$ of the form $L(s,\mathrm{Sym}^j(\pi)\times(\mathrm{Sym}^j(\pi')\otimes\xi))$, where $\xi=\xi^*\otimes|\cdot|^{it_{\xi}}\in\mathfrak{F}_{1}$.  When do these factors have poles?  In order to answer this question, we rely on the fact that $\pi,\pi'\in\mathrm{RA}_F$, in which case $\pi$ and $\pi'$ have associated $\ell$-adic representations for each prime $\ell$.  In order to use these representations, we must restrict $\xi^*$ to be associated to a ray class character, so that its image has finite order.  For a set $U$ of places of $F$, let $\delta(U)$ (resp. $\overline{\delta}(U)$) denote its Dirichlet density (resp. its upper Dirichlet density).

\begin{lemma}
\label{lem:Rajan_multiplicity_one}
Let $F$ be totally real, $\pi,\pi'\in\mathrm{RA}_F$, and $\xi\in\mathfrak{F}_{1}$ with $\xi^*$ corresponding with a ray class character.  Let $j\geq 2$.  If $\pi\not\sim\pi'$, then $L(s,\mathrm{Sym}^j(\pi)\times(\mathrm{Sym}^j(\pi')\otimes\xi))$ is entire.
\end{lemma}
\begin{proof}
We will prove the contrapositive.  Let $F$ be totally real, $j\geq 2$, $\pi,\pi'\in\mathrm{RA}_{F}$, and $\xi=\xi^*|\cdot|^{it_{\xi}}\in\mathfrak{F}_{1}$ with $\xi^*$ corresponding with a ray class character.  It suffices to let $\pi,\pi'\in\mathrm{RA}_F\cap\mathfrak{F}_{2}^*$ and $t_{\xi}=0$.  Now, $\omega_{\pi}$ and $\omega_{\pi'}$ also correspond with ray class characters. We have that $\mathrm{Sym}^j(\pi),\mathrm{Sym}^j(\pi')\otimes\xi\in\mathfrak{F}_{j+1}$ by Theorem \ref{thm:NT}.  If $L(s,\mathrm{Sym}^j(\pi)\times(\mathrm{Sym}^j(\pi')\otimes\xi))$ has a pole, then $\mathrm{Sym}^j(\pi)=\mathrm{Sym}^j(\widetilde{\pi}')\otimes\overline{\xi}$.  Since $\widetilde{\pi}'=\pi'\otimes\overline{\omega}_{\pi'}$ and $\mathrm{Sym}^j(\widetilde{\pi}')=\mathrm{Sym}^j(\pi')\otimes\overline{\omega}_{\pi'}^{j}$, it follows that $\mathrm{Sym}^j(\pi)=\mathrm{Sym}^j(\pi')\otimes\overline{\omega}_{\pi'}^{j}\overline{\xi}$ and $\overline{\delta}(\{v\nmid\infty\colon \mathrm{Sym}^j(\pi_v)=\mathrm{Sym}^j(\pi_v')\otimes(\overline{\omega}_{\pi'}^{j}\overline{\xi})_v\})=1$.  By the Chebotarev density theorem, $\delta(\{v\nmid\infty\colon\textup{$(\overline{\omega}_{\pi'}^{j}\overline{\xi})_v$ is trivial}\})>0$, so
\begin{equation}
	\label{eqn:rajan}
\overline{\delta}(\{v\nmid\infty\colon \mathrm{Sym}^j(\pi_v)=\mathrm{Sym}^j(\pi_v')\})>0.	
\end{equation}

Since $\pi\in\mathrm{RA}_{F}$, if $v\nmid\infty$ and $\ell$ is the residue characteristic of $F_v$, then there corresponds with $\pi_v$ a continuous $\ell$-adic representation $\rho_{\ell,\pi}\colon \mathrm{Gal}(\overline{F}/F)\to\mathrm{GL}_2(F_v)$ such that if $v\notin S_{\pi}$ and $\sigma_v$ denotes the Frobenius conjugacy class at $v$ in $\mathrm{Gal}(\overline{F}/F)/\ker\rho_{\ell,\pi}$, then $\mathrm{Tr}\,\rho_{\ell,\pi}(\sigma_v)=a_{\pi}(v)$ and $\det\rho_{\ell,\pi}(\sigma_v)=\omega_{\pi}(\varpi_v)$.  Consequently, \eqref{eqn:rajan} implies that
\begin{equation}
\label{eqn:rajan2}
\overline{\delta}(\{v\nmid\infty\colon \mathrm{Tr}((\mathrm{Sym}^j\circ\rho_{\ell,\pi})(\sigma_v))=\mathrm{Tr}((\mathrm{Sym}^j\circ\rho_{\ell,\pi'})(\sigma_v))\})>0.
\end{equation}
By the work of Rajan \cite[Corollary 5.1 and Remark 5.1]{Rajan}, \eqref{eqn:rajan2} implies that $\pi\sim\pi'$.
\end{proof}

\section{Proof of Theorem \ref{thm:main}}
\label{sec:proof_main}

\subsection{Case 1: $\pi\not\sim\pi'$} Let $m,n\geq 0$, $M= \max\{m,n\} \geq 1$, and $\chi=\chi^*|\cdot|^{it_{\chi}}\in\mathfrak{F}_{1}$.  Assume that $\chi^*$ corresponds with a ray class character.  Let $y_{\pi}$ and $y_{\pi'}$ be as in \eqref{eqn:y_def}, and let $Y=\max\{y_{\pi},y_{\pi'}\}$.  Consider $\mathcal{D}(s)$ in \eqref{eqn:D_def}.  By Theorem \ref{thm:NT}, the hypotheses of Lemma \ref{lem:AC} are satisfied.  By Lemmata \ref{lem:AC}, \ref{lem:GRC}, and \ref{lem:coeff}, the Dirichlet coefficients $a_{\mathcal{D}}(v^{\ell})$ of $-(\mathcal{D}'/\mathcal{D})(s)$ are non-negative when $v\notin S_{\pi}\cup S_{\pi'}\cup S_{\chi}$, $|a_{\mathcal{D}}(v^{\ell})|\leq (m+4n+5)^2$ otherwise, and $\mathfrak{S}\subseteq S_{\pi}\cup S_{\pi'}\cup S_{\chi}$.  In the notation of Lemma \ref{lem:GHL}, it follows that
\[
\sum_{v\in \mathfrak{S}}\frac{|a_D(v^{\ell})|\log q_v}{q_v}\ll M^2|\mathfrak{S}| \ll M^2 \log(C(\pi)C(\pi')C(\chi)).
\]
  By Lemma \ref{lem:Rajan_multiplicity_one}, we find that the only pole of $\mathcal{D}(s)$ is at $s=1$, with order 3.  By \eqref{eqn:BH} and Lemmata \ref{lem:GHL} and \ref{lem:AC}, there exists an absolute,  effectively computable constant $\Cl[abcon]{GHL2}>0$ such that $\mathcal{D}(s)$ has no real zeros in the interval $[1,\infty)$ and at most 3 real zeros in
\[
I = [1-\Cr{GHL2}/\{M^{2+Y}\log(M C(\pi)C(\pi'))+M^2 \log C(\chi)\},1).
\]
If $L(s,\mathrm{Sym}^m(\pi)\times(\mathrm{Sym}^n(\pi')\otimes\chi))$ has a zero in $I$, then by \eqref{eqn:dual} and \eqref{eqn:D_def}, $\mathcal{D}(s)$ has a zero in $I$ with order at least 4, a contradiction.  Theorem \ref{thm:main} follows.

If $n=0$, then $L(s,\mathrm{Sym}^m(\pi)\times(\mathrm{Sym}^n(\pi')\otimes\chi))=L(s,\mathrm{Sym}^m(\pi)\otimes\chi)$, and we do not need to invoke Lemma \ref{lem:Rajan_multiplicity_one}.  Therefore, we do not need to restrict $\chi^*\in\mathfrak{F}_{1}$ to correspond with a ray class character.  Thus, the proof in the preceding paragraph shows that there exists an absolute, effectively computable constant $\Cl[abcon]{GHL22}>0$ such that if $\xi\in\mathfrak{F}_{1}$, then
\begin{equation}
\label{eqn:SymL}
L(\sigma,\mathrm{Sym}^m(\pi)\otimes\xi)\neq 0,\qquad \sigma\geq 1-\Cr{GHL22}/(m^{2+y_{\pi}}\log(m C(\pi))+m^2\log C(\xi)).	
\end{equation}

\subsection{Case 2: $\pi\sim\pi'$}

If $m,n\geq 1$ and $\psi\in\mathfrak{F}_{1}$ satisfies $\pi'=\pi\otimes\psi$, then by Lemma \ref{lem:CG},
\[
L(\sigma,\mathrm{Sym}^{m}(\pi)\times(\mathrm{Sym}^n(\pi')\otimes\chi)) = \prod_{r=0}^{\min\{m,n\}}L(\sigma,\mathrm{Sym}^{m+n-2r}(\pi)\otimes\chi\psi^n\omega_{\pi}^r).
\]
By \eqref{eqn:SymL}, each factor is non-zero in $I$ except possibly when $m+n$ is even.  When $m+n$ is even and $r=(m+n)/2$, then $L(s,\chi\psi^n\omega_{\pi}^{(m+n)/2})$ is a factor and might have a zero in $I$.  Finally, if $(m+n)/2\in[0,\min\{m,n\}]$, then $m=n$ (so $m+n$ is even).

\section{Proof of Theorem \ref{thm:ST}}
\label{sec:proof_ST}

Throughout, we assume that $F$ is a totally real field, $\pi,\pi'\in\mathrm{RA}_F\cap\mathfrak{F}_{2}^*$, and $\pi\not\sim\pi'$.  Recall $y_{\pi}$ and $y_{\pi'}$ from \eqref{eqn:y_def}.  Let $\zeta$ (resp. $\zeta'$) be a fixed root of unity such that $\zeta^2$ (resp. $\zeta'^2$) lies in the image of $\omega_{\pi}$ (resp. $\omega_{\pi'}$).  Let $Y=\max\{y_{\pi},y_{\pi'}\}$, $\mathcal{Q}=C(\pi)C(\pi')\geq e$, and
\[
\mathcal{V}=\mathcal{V}(\pi,\pi',\zeta,\zeta')=\{v\nmid\infty\colon v\notin S_{\pi}\cup S_{\pi'},~\omega_{\pi,v}(\varpi_v)=\zeta^2,~\omega_{\pi',v}(\varpi_v)={\zeta'}^2\}.
\]
We will assume that $\mathcal{V}$ is infinite.  If $v\in\mathcal{V}$, then
\[
L(s,\pi_v)^{-1}=(1-\alpha_{1,\pi}(v)q_v^{-s})(1-\alpha_{2,\pi}(v)q_v^{-s}) = 1-a_{\pi}(v)q_v^{-s}+\zeta^2 q_v^{-2s}.
\]
There exists $\theta_{v}\in[0,\pi]$ such that if $A_{\pi}(v)=\mathrm{diag}(\zeta e^{i\theta_{v}},\zeta e^{-i\theta_{v}})$, then
\[
L(s,\pi_v)^{-1}=\det(I_{2\times 2}-A_{\pi}(v)q_v^{-s}),\quad \mathrm{tr}\,A_{\pi}(v)=2\zeta\cos\theta_{v}=a_{\pi}(v),\quad \det A_{\pi}(v)=\zeta^2.
\]
Similarly, if $v\notin S_{\pi'}^{\infty}$ satisfies $\omega_{\pi',v}(\varpi_v)={\zeta'}^2$, then we define the angle $\theta_{v}'\in[0,\pi]$ associated to $\pi_v'$.  Given a set $U$, let $\mathbf{1}_U$ be its indicator function.

Let $I_1$ and $I_2$ be subintervals of $[0,\pi]$, and define $\mathrm{d}\widetilde{\mu}_{\mathrm{ST}}=(2/\pi)(\sin\theta)^2\mathrm{d}\theta$.  We will prove Theorem \ref{thm:ST}(2) in the following equivalent form:  If $x\geq 2$, then
\begin{equation}
\label{eqn:ST_2.0}
\begin{aligned}
&\Big|\sum_{\substack{v\in\mathcal{V} \\ q_v\leq x}}\mathbf{1}_I(\theta_{v})\mathbf{1}_{I'}(\theta_{v}')\Big\slash\sum_{\substack{v\in \mathcal{V} \\ q_v\leq x}}1-\widetilde{\mu}_{\mathrm{ST}}(I)\widetilde{\mu}_{\mathrm{ST}}(I')\Big|\\
&\leq \Cr{ST1} \sqrt{[F:\mathbb{Q}]}\Big(\frac{\log(C(\pi)C(\pi')\log x)}{\sqrt{\log x}}\Big)^{\frac{2}{2+\max\{y_{\pi},y_{\pi'}\}}}.
\end{aligned}
\end{equation}
We will not give the details for Theorem \ref{thm:ST}(1) because the proof is far less complicated.

Recall $\Cr{main}$ from  Corollary \ref{cor:big_ZFR2}.  Without loss of generality, we assume that $\Cr{main}\leq 1$.
\begin{lemma}
\label{lem:CDT}
There exists an absolute, effectively computable constant $\Cl[abcon]{CDT}\geq 1$ such that
\begin{equation}
\label{eqn:x_range}
\Big(\sum_{\substack{v\in\mathcal{V} \\ q_v\leq x}}1\Big)^{-1}\ll \mathcal{Q}^{\Cr{CDT}}\frac{\log x}{x},\qquad x\geq \exp\Big(\Big(\frac{16\Cr{CDT}+224}{\Cr{main}}\Big)^4[F:\mathbb{Q}]^{1+\frac{Y}{2}}(\log\mathcal{Q})^2\Big).	
\end{equation}
\end{lemma}
\begin{proof}
By the proof of \cite[Theorem 3.1]{TZ_least}, there exists an absolute, effectively computable constant $\Cl[abcon]{CDT2}>0$ such that left hand side of \eqref{eqn:x_range} is $\ll (\mathcal{Q}|\mathrm{im}(\omega_{\pi})|\cdot|\mathrm{im}(\omega_{\pi'})|)^{\Cr{CDT2}}(\log x)/x$ in the prescribed range of $x$.  It remains to bound $|\mathrm{im}(\omega_{\pi})|$ and $|\mathrm{im}(\omega_{\pi'})|$.  By \cite[Lemma 2.12]{TZ_least} we have that $|\mathrm{im}(\omega_{\pi})|\leq e^{O([F:\mathbb{Q}])}C(\omega_{\pi})$.  By \cite[Theorem A]{RamakrishnanYang}, we have that $C(\omega_{\pi})\leq C(\pi)\leq \mathcal{Q}$.  Since $[F:\mathbb{Q}]\ll\log\mathcal{Q}$, the lemma follows.
\end{proof}

Outside of the range of $x$ in \eqref{eqn:x_range}, Theorem \ref{thm:ST} is trivial.  In this range, we have that
\begin{equation}
\label{eqn:M_def}
M=\frac{1}{\sqrt{[F:\mathbb{Q}]}}\Big(\frac{\sqrt{\Cr{main}\log x}}{(16\Cr{CDT}+224)\log(\mathcal{Q}\log x)}\Big)^{\frac{2}{2+Y}}\geq 2.
\end{equation}

\subsection{A prime number theorem}

Let $m,n\geq 0$ be integers.  We will use Corollary \ref{cor:big_ZFR2} to prove a prime number theorem for the $L$-functions $L(s,\mathrm{Sym}^m(\pi)\times(\mathrm{Sym}^n(\pi')\otimes\chi))$.

\begin{proposition}
\label{prop:PNT}
Let $\chi\in\mathfrak{F}_{1}^*$ correspond with a ray class character and $\mathcal{M}= \max\{m,n\}$.  If $x\geq 2$ and $\mathcal{M}\geq 1$, then
\begin{align*}
&\sum_{\substack{v\notin S_{\pi}\cup S_{\pi'}\cup S_{\chi} \\ q_v\leq x}}a_{\mathrm{Sym}^m(\pi)\times(\mathrm{Sym}^n(\pi')\otimes\chi)}(v)\\
&\ll x\exp\Big(-\frac{(\Cr{main}/8)\log x}{\mathcal{M}^{2+Y}\log(\mathcal{Q}C(\chi)\mathcal{M})+\mathcal{M}\sqrt{\Cr{main} [F:\mathbb{Q}]\log x}}\Big) (\mathcal{M}^{2+Y}\log(\mathcal{Q}C(\chi)\mathcal{M}x))^6.
\end{align*}
\end{proposition}
\begin{proof}
At most $[F:\mathbb{Q}]$ nonarchimedean places of $F$ lie above any rational prime. Using the Brun--Titchmarsh theorem, we find that
\begin{equation}
\label{eqn:BT}
\#\{v\nmid\infty\colon q_v\in(x,x+h]\}\ll [F:\mathbb{Q}]h/\log h,\qquad x>0,\quad h>1.
\end{equation}
By Corollary \ref{cor:big_ZFR2}, Lemma \ref{lem:GRC}, and \eqref{eqn:BT}, if $x\geq 2$, then we can proceed as in \cite[Theorem 5.13]{IK} to obtain
\begin{align*}
&\sum_{\substack{\ell\geq 1,~q_v^{\ell}\leq x}}a_{\mathrm{Sym}^m(\pi)\times(\mathrm{Sym}^n(\pi')\otimes\chi)}(v^{\ell})\log q_v\\
&\ll x\exp\Big(-\frac{(\Cr{main}/4)\log x}{\mathcal{M}^{2+Y}\log(\mathcal{Q}C(\chi)\mathcal{M})+\mathcal{M}\sqrt{\Cr{main} [F:\mathbb{Q}]\log x}}\Big) (\mathcal{M}^{2+Y}\log(\mathcal{Q}C(\chi)\mathcal{M}x))^6.
\end{align*}
By Lemma \ref{lem:GRC} and \eqref{eqn:BT}, the contribution from $\ell\geq 2$ or $v\in S_{\pi}\cup S_{\pi'}\cup S_{\chi}$ is negligible.  The proposition now follows by partial summation.
\end{proof}
\begin{corollary}
\label{cor:PNT}
Let $\chi\in\mathfrak{F}_{1}^*$ correspond with a ray class character, and assume that $S_{\chi}\subseteq S_{\pi}\cup S_{\pi'}$ and $C(\chi)\leq \mathcal{Q}$.  If \eqref{eqn:x_range} and \eqref{eqn:M_def} are satisfied, then
\[
\sum_{\substack{0\leq m,n\leq M \\ (m,n)\neq(0,0)}}\sum_{\substack{v\notin S_{\pi}\cup S_{\pi'} \\ q_v\leq x}}a_{\mathrm{Sym}^m(\pi)\times(\mathrm{Sym}^n(\pi')\otimes\chi)}(v)\ll\frac{x}{(\mathcal{Q}\log x)^{\Cr{CDT}+1}}.
\]
\end{corollary}
\begin{proof}
This follows from Proposition \ref{prop:PNT} and our hypotheses on $\chi$, $x$, and $M$.
\end{proof}

\subsection{Proof of \eqref{eqn:ST_2.0}}

The Chebyshev polynomials of the second kind, defined by the relation $U_m(\cos\theta) = \sin((m+1)\theta)/(\sin\theta)$, form an orthonormal basis of $L^2([0,\pi],\widetilde{\mu}_{\mathrm{ST}})$.  We approximate $\mathbf{1}_{I}(\theta_{v})\mathbf{1}_{I'}(\theta_{v}')$ with respect to this basis as follows.

\begin{lemma}
\label{lem:ET1}
Let $I,I'\subseteq[0,\pi]$ be subintervals.  For each $M\geq 1$, there exist functions
\[
T_{I,I',M}^{\pm}(\theta,\theta')=\sum_{0\leq m\leq M}~\sum_{0\leq n\leq M}\alpha_{I,I',M}^{\pm}(m,n)U_m(\cos \theta)U_n(\cos \theta')
\]
such that if $\theta,\theta'\in[0,\pi]$, then $T_{I,I',M}^{-}(\theta,\theta')\leq \mathbf{1}_{I}(\theta)\mathbf{1}_{I'}(\theta')\leq T_{I,I',M}^{+}(\theta,\theta')$.  There holds
	\[
	\left|\alpha_{I,I',M}^{\pm}(m,n)-\begin{cases}
	\mu_{\mathrm{ST}}(I)\mu_{\mathrm{ST}}(I')&\mbox{if $m=n=0$,}\\
	0&\mbox{otherwise}
	\end{cases}\right|\ll\begin{cases}
	1/M&\mbox{if $m=n=0$,}\\
	1/m&\mbox{if $m\neq 0$ and $n=0$,}\\
	1/n&\mbox{if $m=0$ and $n\neq 0$,}\\
	1/(mn)&\mbox{otherwise.}
	\end{cases}
	\]
\end{lemma}
\begin{proof}
For $j=1,2$, let $I_j=[a_j,b_j]\subseteq[0,\pi]$ and $J_j=[a_j/(2\pi),b_j/(2\pi)]$.  Given an interval $J\subseteq[0,1]$, the Fourier expansion of $\mathbf{1}_J(\theta)$ is $\mathbf{1}_J(\theta)=\sum_{n\in\mathbb{Z}}\widehat{\mathbf{1}}_J(n)e^{2\pi i n \theta}$, where $\widehat{\mathbf{1}}_J(n)=\int_J e^{-2\pi i n t}dt$. Cochrane \cite[Theorem 1]{Cochrane} proved that there exist functions
\[
S_{J_1,J_2,M}^{\pm}(x_1,x_2)=\sum_{\substack{m_1,m_2\in\mathbb{Z} \\ 0\leq |m_1|,|m_2|\leq M}}\widehat{S}_{J_1,J_2,M}^{\pm}(m_1,m_2)e^{2\pi i(m_1 x_1+m_2 x_2)}
\]
with the following properties.
\begin{enumerate}
	\item If $(x_1,x_2)\in[0,1]^2$, then $S_{I_1,I_2,M}^-(x_1,x_2)\leq \mathbf{1}_{J_1}(x_1)\mathbf{1}_{J_2}(x_2)\leq S_{I_1,I_2,M}^+(x_1,x_2)$.
	\item The identity $\widehat{S}_{J_1,J_2,M}^{\pm}(m_1,m_2)+\widehat{S}_{J_1,J_2,M}^{\pm}(-m_1,-m_2)=2\mathrm{Re}(\widehat{S}_{J_1,J_2,M}^{\pm}(m_1,m_2))$ holds.
	\item If $0\leq |m_1|,|m_2|\leq M$, then $|\widehat{S}_{J_1,J_2,M}^{\pm}(m_1,m_2)-\widehat{1}_{J_1}(m_1)\widehat{1}_{J_2}(m_2)|\ll 1/M$.
\end{enumerate}

The desired functions $T_{I_1,I_2,M}^{\pm}(\theta,\theta')$ are the even functions 
\begin{align*}
\sum_{t_1\in\{-1,1\}}\sum_{t_2\in\{-1,1\}}S_{J_1,J_2,M}^{\pm}\Big(t_1\frac{\theta}{2\pi},t_2\frac{\theta'}{2\pi}\Big)
\end{align*}
Once we introduce $U_{-1}(\cos\theta)=0$ and $U_{-2}(\cos\theta)=-1$, we express $T_{I_1,I_2,M}^{\pm}$ in the Chebyshev basis using the identity $U_m(\cos\theta)-U_{m-2}(\cos\theta)=2\cos(m\theta)$ for $m\geq 0$.
\end{proof}


\begin{corollary}
\label{cor:ET2}
If $x\geq 2$ and $I,I'\subseteq[-1,1]$ are subintervals, then
\begin{equation}
\label{eqn:ET2}
\begin{aligned}
&\Big|\sum_{\substack{v\in\mathcal{V} \\ q_v\leq x}}\mathbf{1}_I(\theta_{v})\mathbf{1}_{I'}(\theta_{v}')\Big\slash\sum_{\substack{v\in\mathcal{V} \\ q_v\leq x}}1-\widetilde{\mu}_{\mathrm{ST}}(I)\widetilde{\mu}_{\mathrm{ST}}(I')\Big|\\
&\ll \frac{1}{M}+\sum_{\substack{0\leq m,n\leq M \\ (m,n)\neq (0,0)}}  \Big|\sum_{\substack{v\in\mathcal{V} \\ q_v\leq x}}a_{\mathrm{Sym}^m(\pi)\times\mathrm{Sym}^n(\pi')}(v)\Big|\Big\slash\sum_{\substack{v\in\mathcal{V} \\ q_v\leq x}}1.
\end{aligned}
\end{equation}
\end{corollary}
\begin{proof}
This is a straightforward consequence of Lemma \ref{lem:ET1} and the fact that if $v\in\mathcal{V}$, then $U_m(\cos\theta_{v})U_n(\cos\theta_{v}')=\zeta^{-m} \zeta'^{-n} a_{\mathrm{Sym}^m(\pi)\times \mathrm{Sym}^n(\pi')}(v)$.
\end{proof}

If $G$ is the finite abelian group of characters generated by $\omega_{\pi}$ and $\omega_{\pi'}$, then the condition $v\in\mathcal{V}$ can be expressed using orthogonality of characters in $G$.  If $\chi\in G$, then there exist integers $a,b\geq 0$ such that $\chi=\omega_{\pi}^a \omega_{\pi'}^b$.  Since $\widetilde{\pi}=\pi\otimes\overline{\omega}_{\pi}$ and $\widetilde{\pi}'=\pi'\otimes\overline{\omega}_{\pi'}$,  we have that $S_{\chi}\subseteq S_{\pi}\cup S_{\pi'}$ and $C(\chi)\leq \mathcal{Q}$ per \cite[Theorem A]{RamakrishnanYang}.  We now observe that
\begin{equation}
\label{eqn:char_orth}
\Big|\sum_{\substack{v\in\mathcal{V} \\ q_v\leq x}}a_{\mathrm{Sym}^m(\pi)\times\mathrm{Sym}^n(\pi')}(v)\Big|\ll \max_{\chi\in G}\Big|\sum_{\substack{v\notin S_{\pi}\cup S_{\pi'} \\ q_v\leq x}}a_{\mathrm{Sym}^m(\pi)\times(\mathrm{Sym}^n(\pi')\otimes\chi)}(v)\Big|.
\end{equation}
Therefore, by Lemma \ref{lem:CDT}, Corollary \ref{cor:PNT}, and \eqref{eqn:char_orth}, we find that \eqref{eqn:ET2} is, as desired,
\[
\ll \frac{1}{M}+\frac{\mathcal{Q}^{\Cr{CDT}}}{(\mathcal{Q}\log x)^{\Cr{CDT}}}\ll \sqrt{[F:\mathbb{Q}]}\Big(\frac{\log(\mathcal{Q}\log x)}{\sqrt{\log x}}\Big)^{\frac{2}{2+Y}}.
\]
\bibliographystyle{abbrv}
\bibliography{JAThorner_SiegelZeros}

\end{document}